\documentclass{amsart}

\usepackage{a4wide,amsmath,amssymb}
\usepackage[toc,page]{appendix}
\usepackage{hyperref}
\usepackage{url}
\usepackage{courier}
\usepackage{mathtools}

\usepackage{xcolor}

\newtheorem{thm}{Theorem}[section]
\newtheorem{lemma}[thm]{Lemma}

\theoremstyle{remark}
\newtheorem{rem}[thm]{Remark}

\theoremstyle{definition}

\newtheoremstyle{Claim}{}{}{\itshape}{}{\itshape\bfseries}{:}{ }{#1}
\theoremstyle{Claim}

\newcommand{\R}{\mathbb{R}}

\makeatletter
\@namedef{subjclassname@2020}{\textup{2020} Mathematics Subject Classification}
\makeatother

\theoremstyle{plain}

\def\sideremark#1{\ifvmode\leavevmode\fi\vadjust{
\vbox to0pt{\hbox to 0pt{\hskip\hsize\hskip1em
\vbox{\hsize3cm\tiny\raggedright\pretolerance10000
\noindent #1\hfill}\hss}\vbox to8pt{\vfil}\vss}}}

\begin{document}

\title[]{On the strong maximum principle for fully nonlinear parabolic equations of second order}

\author{Alessandro Goffi} 

\date{\today}
 \thanks{The author is member of the Gruppo Nazionale per l'Analisi Matematica, la Probabilit\`a e le loro Applicazioni (GNAMPA) of the Istituto Nazionale di Alta Matematica (INdAM). The author was partially supported by the INdAM-GNAMPA Project 2023 ``Problemi variazionali/nonvariazionali: interazione tra metodi integrali e principi del massimo''  and by the King Abdullah University of Science and Technology (KAUST) project CRG2021-4674 ``Mean-Field Games: models, theory and computational aspects".}

\address{Dipartimento di Matematica ``Tullio Levi-Civita", Universit\`a degli Studi di Padova, Via Trieste 63, 35121 Padova, Italy} \email{alessandro.goffi@unipd.it}

\maketitle
\begin{abstract}
We provide a proof of strong maximum and minimum principles for fully nonlinear uniformly parabolic equations of second order. The approach is of parabolic nature, slightly differs from the earlier one proposed by L. Nirenberg and does not exploit the parabolic Harnack inequality. 
\end{abstract}

\section{Introduction}
\label{intro}
We prove here a strong maximum principle for fully nonlinear uniformly parabolic equations of the form
\begin{equation}\label{eqintro}
F(x,t,u,Du,D^2u)-\partial_t u=0\text{ in }D\subset\R^{n+1},
\end{equation}
where $F:D\times\R\times\R^n\times\mathcal{S}_n\to\R$, $\mathcal{S}_n$ being the space of $n\times n$ symmetric matrices, is assumed to satisfy the following structure condition: there exist $b\geq0$, $c\leq0$ continuous such that for $M,N\in\mathcal{S}_n$, $N\geq0$, and every $p,q\in\R^n$, $(x,t)\in D$ we have
\begin{equation}\label{upar}
\lambda \|N\|- b|p-q|+c(r-s)\leq F(x,t,r,p, M+N) -F(x,t,s,q,N)\leq \Lambda\|N\|+ b|p-q|+c(r-s)
\end{equation}
and $F(x,t,0,0,0)=0$ on $D$, $\|N\|=\mathrm{Tr}(N)$. \\
By strong maximum (resp. minimum) principle for an evolution operator $\mathcal{P}u=0$ in $\Omega_T:=\Omega\times(0,T)$ we mean the following property:\\
\begin{center}
\textit{Any continuous viscosity subsolution (resp. supersolution) to $\mathcal{P}u=0$ in $\Omega_T$ that attains a nonnegative maximum (minimum) at $(x_0,t_0)\in \Omega_T$ is constant in $\overline{\Omega}\times(0,t_0)$.}
\end{center}
\medskip
L. Nirenberg \cite{Nirenberg53}, see also \cite{FriedmanBook}, proved the strong maximum principle for classical solutions of linear parabolic equations, extending the proof due to E. Hopf. Then, this was proved for fully nonlinear parabolic equations in \cite{DaLiopar,CLN}. The proof of the vertical propagation of maximum points in \cite{DaLiopar} uses the horizontal/elliptic propagation on suitable small rectangles, cf. Corollary 2.3 in \cite{DaLiopar}. As for the elliptic strong maximum principle for viscosity solutions of linear and fully nonlinear, even degenerate, equations we refer to \cite{Calabi,BDL1,BG1,HL,GP} and the references therein. We mention that F. Da Lio \cite{DaLiopar} proved the result for rather more general equations than \eqref{eqintro}, even degenerate, following the route of \cite{Nirenberg53,FriedmanBook} in the context of viscosity solutions. The paper \cite{CLN} used the argument in \cite{Nirenberg53}, proving the result when $F\in C^1$ in all arguments in the form of a strong comparison principle: in this case $F$ was assumed strictly parabolic only and the derivatives $|F_p|,|F_r|$ are not necessarily uniformly bounded.\\
Here, our proof follows the approach of \cite{IKOreprint,Landis} valid for classical solutions of linear parabolic nondivergence structure equations and slightly differs from that of L. Nirenberg \cite{Nirenberg53}, L. Caffarelli-Y. Li-L. Nirenberg \cite{CLN} and F. Da Lio \cite{DaLiopar}. The main difference with respect to the aforementioned works is the choice of a space-time barrier-like function that allows to show the vertical propagation of maximum points on inclined cylinders whose axes are broken lines starting from the maximum points. Other different proofs of the strong maximum principle for linear and nonlinear parabolic equations exploit the weak Harnack inequality \cite{Lieberman,ImbertSilvestre}, see also Proposition 4.9 in \cite{CC} for the elliptic case. \\
We emphasize the parabolic nature of our approach, since it does not exploit an elliptic argument and works directly on parabolic cylinders. Moreover, since it does not make use of the parabolic Harnack inequality as in \cite{CC,ImbertSilvestre}, such a proof could be particularly useful in those settings where the latter tool is not available in full generality: this is the case of fully nonlinear subelliptic equations \cite{BG1} and those posed on Riemannian manifolds \cite{GP}. \\
We conclude the proof of the main result with some remarks about the applications of the strong maximum principle to various qualitative properties of nonlinear evolution equations, such as the strong comparison principle for fully nonlinear parabolic equations, a Tychonoff uniqueness principle for Hessian evolution equations and a positivity result. Notably, the result also gives a proof of the strong maximum principle for fully nonlinear elliptic equations \cite{Calabi,BDL1,BG1,GP} by means of a different method. Some examples to which the results apply will be presented in the last section of the manuscript.
\section{Notations and preliminaries}
Let $b,c,f$ be continuous and bounded functions in $D\subset\R^{n+1}$, $b\geq0$, $c\leq0$ and $\lambda\leq\Lambda$ be two positive constants (the ellipticity constants). \\
We denote for $M\in\mathcal{S}_n$ by
\[
\mathcal{M}^+_{\lambda,\Lambda}(M)=\sup_{\lambda I_n\leq A\leq \Lambda I_n}\mathrm{Tr}(AM)=\Lambda \sum_{e_i>0}e_i(M)+\lambda\sum_{e_i<0}e_i(M),
\]
\[
\mathcal{M}^-_{\lambda,\Lambda}(M)=\inf_{\lambda I_n\leq A\leq \Lambda I_n}\mathrm{Tr}(AM)=\lambda \sum_{e_i>0}e_i(M)+\Lambda\sum_{e_i<0}e_i(M),
\]
$e_i$ being the eigenvalues of $M$, the Pucci's extremal operators. We denote by $\underline{S}(\lambda,\Lambda,b,c,f)$ the space of continuous functions $u\in C(D)$ such that 
\[
\mathcal{M}^+_{\lambda,\Lambda}(D^2u)+b(x,t)|Du|+c(x,t)u-\partial_t u\geq f(x,t)
\]
 in the viscosity sense in $Q$. Similarly, we denote by $\overline{S}(\lambda,\Lambda,b,c,f)$ the space of continuous functions $u\in C(D)$ such that 
 \[
 \mathcal{M}^-_{\lambda,\Lambda}(D^2u)-b(x,t)|Du|+c(x,t)u-\partial_t u\leq f(x,t).
 \]
 Note that if $F$ in \eqref{eqintro} satisfies \eqref{upar} and $F(x,t,0,0,0)=0$, then it satisfies in the viscosity sense
 \[
 \mathcal{M}^+_{\lambda,\Lambda}(D^2u)+b(x,t)|Du|+c(x,t)u-\partial_t u\geq0\geq  \mathcal{M}^-_{\lambda,\Lambda}(D^2u)-b(x,t)|Du|+c(x,t)u-\partial_t u.
 \]
 For more details on these classes and the proof of the results we refer to \cite{CC}.\\
We follow the notation in \cite{Landis}. We denote by $\mathbb{Q}^{t_1,t_2}_{x_0,R}:=\{(x,t)\in \R^{n+1}:|x-x_0|<R,t_1<t<t_2\}$ a cylinder in $\R^{n+1}$. Moreover, $S^{t_1,t_2}_{x_0,R}$ is the lateral (closed) surface, $Q^{t_1}_{x_0,R}$ is the lower (open) base (i.e. the $n$-dimensional open ball of radius $R$ centered at $(x_0,t_1)$ on the hyperplane $t=t_1$) and $Q^{t_2}_{x_0,R}$ is the upper (open) base of the cylinder $\mathbb{Q}^{t_1,t_2}_{x_0,R}$. Here, the $t$-axis is directed upward.\\
Let $D\subset \R^{n+1}$ be a bounded domain. By the upper base of the domain $D$, denoted with $\gamma(D)$, we mean the set of those points $(x,t)\in \partial D$ satisfying the following property: for each point of $\gamma(D)$ there exists $h>0$ such that the cylinder $\mathbb{Q}^{t-h,t}_{x,h}$ belongs to $D$, while the cylinder $\mathbb{Q}^{t,t+h}_{x,h}$ lies outside $D$. Let $(x_0,t_0)\in\overline{D}$. A set $D'$ is called subdomain of $D$ subordinate to $(x_0,t_0)$ if $(x_0,t_0)$ can be joined with every point of $D'$ by a broken line that lies  in $D$ (except at most $(x_0,t_0)$), is uniquely projected to the $t$-axis (i.e. two points of the broken line are projected to different points of the $t$-axis), and has $(x_0,t_0)$ as upper end-point.
\section{Main result}
We consider the following equation
\begin{equation}\label{maineq}
F(x,t,u,Du,D^2u)-\partial_t u=0\text{ in }D.
\end{equation}
The next result is a fully nonlinear version of Theorem 2.3 in Chapter 3 of \cite{Landis}, cf. also Theorem 6 in \cite{IKOreprint}, and applies to viscosity solutions. Note also that the next property holds for a general parabolic domain in $\R^{n+1}$ rather than a cylinder of the form $\Omega\times(0,T)$.
\begin{thm}[Strong maximum/minimum principles]\label{mainSMP}
Let $D$ be a domain in $\R^{n+1}$ with upper base $\gamma(D)$. Let $u$ be a continuous viscosity subsolution (resp. supersolution) in $D\cup\gamma(D)$ to \eqref{maineq}, with $F$ satisfying \eqref{upar}, such that $c(x,t)\leq0$, $b\geq0$ and $b,c$ are bounded,  attaining a nonnegative maximum (nonpositive minimum) at some point $(x_0,t_0)\in D\cup\gamma(D)$. Then, $u$ is constant in a subdomain $D'$ of $D$ subordinate to the point $(x_0,t_0)$.
\end{thm}
To prove the strong maximum principle we need some preliminary lemmas. 
The next is the crucial step towards the propagation of maximum points vertically in the parabolic cylinder.
\begin{lemma}\label{axis}
Under the same assumptions of Theorem \ref{mainSMP}, let $u$ be a continuous viscosity subsolution (resp. supersolution) of \eqref{maineq} in the cylinder $\mathbb{Q}^{t_1,t_2}_{x_0,R}$ and on the upper base $Q^{t_2}_{x_0,R}$. If $u$ attains the nonnegative maximum (nonpositive minimum) $M$ at $(x_0,t_2)$, then the maximum (minimum) propagates, namely $u=M$, on the axis of the cylinder.
\end{lemma}
\begin{proof}
We prove the statement for viscosity subsolutions, the other being similar using that $u\in\underline{\mathcal{S}}$. First, since $F$ is uniformly parabolic, it is enough to prove the claim for viscosity solutions belonging to $\underline{S}$, i.e. to the viscosity inequality
\[
\mathcal{M}^+_{\lambda,\Lambda}(D^2u(x,t))+b(x,t)|Du|+c(x,t)u-\partial_t u\geq0.
\]
 Suppose by contradiction that there exists a point $(x_0,t')$ on the axis of the cylinder at which $u<M-\alpha$ for some $\alpha>0$. By continuity, we can find $r_0$ small enough such that $r_0<\min\{R,1\}$ and $u(x,t')<M-\alpha$ for $|x-x_0|<r_0$. In the inner closed cylinder $\mathbb{Q}^{t',t_2}_{x_0,r_0}$ we consider the function
\[
v(x,t)=M-\alpha e^{-\beta(t-t')}[r_0^2-|x-x_0|^2]^2.
\]
We assume for the moment that there exists $\beta>0$ such that $v$ is a classical (strict) supersolution to 
\[
\mathcal{M}^+_{\lambda,\Lambda}(D^2u(x,t))+b(x,t)|Du|+c(x,t)u-\partial_t u=0.
\]
We first show how to obtain a contradiction and conclude the proof. We first observe that $u-v$ is nonpositive on the lower base and the lateral surface of the cylinder since $v>M-\alpha$ on the lower base and equals $M$ on the lateral surface $S^{t',t_2}_{x_0,R}$. By the comparison principle, see e.g. Theorem 8.2 in \cite{CIL}, $u\leq v$ in the cylinder, and we have that
\[
M=u(x_0,t_2)\leq v(x_0,t_2)=M-\alpha r_0^2e^{-\beta(t_2-t')}<M,
\]
which gives the desired contradiction.\\
\par\smallskip
We are now left to prove the existence of $\beta$. We define $\varphi(x)=[r_0^2-|x-x_0|^2]^2$. Then
\[
\partial_{x_i}\varphi(x)=-4[r_0^2-|x-x_0|^2](x_i-(x_0)_i),
\]
\[
D^2\varphi(x)=8|x-x_0|^2\frac{x-x_0}{|x-x_0|}\otimes \frac{x-x_0}{|x-x_0|}-4[r_0^2-|x-x_0|^2]I_n,
\]
where $(x\otimes x)_{ij}=x_ix_j$. We recall that the eigenvalues of a matrix $M$ of the form
\[
M=\nu I_n+\xi v\otimes v,\ v\in\R^n,\ |v|=1,\ \nu,\xi\in\R,
\]
are $\nu$ with multiplicity $n-1$ and $\nu+\xi$. Therefore, the eigenvalues of $D^2\varphi$ are $e_1(D^2\varphi)=-4[r_0^2-|x-x_0|^2]$ with multiplicity $n-1$ and $e_2(D^2\varphi)=8|x-x_0|^2-4[r_0^2-|x-x_0|^2]=12|x-x_0|^2-4r_0^2$, which is simple. In particular, $e_1\leq0$. As for $e_2$, note that when $|x-x_0|\leq \frac{r_0^2}{3}$ we have $e_2\leq0$, hence
\begin{align*}
\mathcal{M}^+_{\lambda,\Lambda}(D^2\varphi)&=\lambda\{-4(n-1)[r_0^2-|x-x_0|^2]+8|x-x_0|^2-4[r_0^2-|x-x_0|^2\}\\
&= \lambda\{-4n[r_0^2-|x-x_0|^2]+8|x-x_0|^2\}.
\end{align*}
Thus, by the properties of Pucci's extremal operators we get
\[
\mathcal{M}^+_{\lambda,\Lambda}(D^2v)=-\alpha e^{-\beta(t-t')}\{8\lambda |x-x_0|^2-4n\lambda[r_0^2-|x-x_0|^2]\}.
\]
If instead $|x-x_0|\geq \frac{r_0^2}{3}$ (and $|x-x_0|<r_0$) we have $e_2\geq0$, and we conclude
\begin{align*}
\mathcal{M}^+_{\lambda,\Lambda}(D^2\varphi)&=\lambda\{-4(n-1)[r_0^2-|x-x_0|^2]\}+\Lambda\{8|x-x_0|^2-4[r_0^2-|x-x_0|^2\}\\
&= 8\Lambda|x-x_0|^2-[4\lambda(n-1)+4\Lambda]\{r_0^2-|x-x_0|^2\}.
\end{align*}
Then
\[
\mathcal{M}^+_{\lambda,\Lambda}(D^2v)=-\alpha e^{-\beta(t-t')}\{8\Lambda|x-x_0|^2-[4\lambda(n-1)+4\Lambda](r_0^2-|x-x_0|^2)\}.
\]
In both cases we end up with
\[
\mathcal{M}^+_{\lambda,\Lambda}(D^2v)\leq -\alpha e^{-\beta(t-t')}(8\lambda|x-x_0|^2-c_{n,\lambda,\Lambda}\{r_0^2-|x-x_0|^2\})
\]
for $c_{n,\lambda,\Lambda}:=4\lambda(n-1)+4\Lambda$. Since
\[
\partial_t v=\beta\alpha e^{-\beta(t-t')}[r_0^2-|x-x_0|^2]^2
\]
we get, using $c\leq0$,
\begin{align*}
&\mathcal{M}^+_{\lambda,\Lambda}(D^2v)+b(x,t)|Dv(x,t)|+c(x,t)v(x,t)-\partial_t v\\
&\leq Mc -\alpha e^{-\beta(t-t')}\{8\lambda|x-x_0|^2-r(x,t)[r_0^2-|x-x_0|^2]+\beta[r_0^2-|x-x_0|^2]^2\}\\
&\leq -\alpha e^{-\beta(t-t')}\{8\lambda|x-x_0|^2-K[r_0^2-|x-x_0|^2]+\beta[r_0^2-|x-x_0|^2]^2\}=:-\alpha e^{-\beta(t-t')}\Psi(|x-x_0|),
\end{align*}
where $r(x,t)$ is a bounded function, i.e. $|r(x,t)|\leq K$, for some $K>0$ depending on $n,\lambda,\Lambda$ and on the coefficients $b,c$, due to their boundedness assumptions. For $|x-x_0|=r_0$, we have 
\[
8\lambda|x-x_0|^2-K[r_0^2-|x-x_0|^2]=8\lambda r_0^2>0\ .
\]
Therefore, there is $\delta>0$ such that $\Psi$ is positive for any $\beta>0$ if $r_0^2-|x-x_0|^2\leq\delta$. Fixing such a $\delta$, we take then $\beta$ large so that $\Psi>0$ when  $r_0^2-|x-x_0|^2>\delta$.
\end{proof}
We now prove that the maximum point propagates along the axes of inclined cylinders.
\begin{lemma}\label{incl}
Consider the inclined cylinder $D^I:=\{(x,t): |x-[x_0+\eta(t-t_1)]|<R,\ t_1<t\leq t_2\}$, $\eta\in\R^n$. Let $u$ be a continuous viscosity subsolution (resp. supersolution) of \eqref{maineq} up to the upper base of $D^I$. If $u$ attains the nonnegative maximum (nonpositive minimum) $M$ at $(x_0+\eta(t_2-t_1),t_2)$. Then $u=M$ at the points $x=x_0+\eta(t-t_1)$, $t\in(t_1,t_2)$.
\end{lemma}
\begin{proof}
It is enough to make the change of variables
\[
\tilde t=t\ ,\tilde x=x-\eta(t-t_1),
\]
which leads to the viscosity inequality
\[
\mathcal{M}^+_{\lambda,\Lambda}(\widetilde{D}^2u)+b|\widetilde{D}u|-\eta\cdot \widetilde{D}u+cu-\partial_{\tilde t}u\geq0\ ,
\]
where $\widetilde{D},\widetilde{D}^2$ stands for the gradient and Hessian operators with respect to the variable $\tilde x$, which has the same form of the one treated in the previous Lemma \ref{axis} with the addition of a term involving a first-order power of the quantity $r_0^2-|x-x_0|^2$. The presence of this new term does not affect the proof of Lemma \ref{axis}. Then, the inclined cylinder becomes $t_1<t<t_2$, $|\tilde x-x_0|<R$. We are thus in position to apply Lemma \ref{axis} from which $u\equiv M$ on $t_1<t<t_2$ and $x=x_0+\eta(t-t_1)$.
\end{proof}
\begin{proof}[Proof of Theorem \ref{mainSMP}] It is only geometrical and the same of the linear case. If $u$ attains its maximum at $(x_0,t_0)\in D\cup\gamma(D)$, assume that the point $(x',t')$ can be joined to $(x_0,t_0)$ by a broken line in $D\cup\gamma(D)$ having $(x_0,t_0)$ as upper end-point. For every segment of the broken line one can construct a cylinder, which is in general inclined as in Lemma \ref{incl}, such that its lower base is orthogonal to the $t$-axis, the broken line is the axis of the cylinder and it belongs to $D\cup\gamma(D)$. We then apply repeatedly Lemma \ref{incl} to these cylinders, starting from the upper one, to find $u\equiv M$ at every point of the broken line, and hence at $(x',t')$.
\end{proof}
In the case of an equation without lower-order terms of the form
\begin{equation}\label{simpleeq}
F(x,t,D^2u)-\partial_t u=0\text{ in }D\subset\R^{n+1},
\end{equation}
in the statement of Theorem \ref{mainSMP} we can remove the sign property of the maximum/minimum point by simply saying that it attains a maximum/minimum, as the operator does not change by the addition of a constant. The result then reads as follows:
\begin{thm}\label{simplemain}
Let $D$ be a domain in $\R^{n+1}$ with upper base $\gamma(D)$. Let $u$ be a continuous viscosity subsolution (resp. supersolution) in $D\cup\gamma(D)$ to \eqref{simpleeq}, with $F$ satisfying \eqref{upar}  and attaining a maximum (minimum) at some point $(x_0,t_0)\in D\cup\gamma(D)$. Then, $u$ is constant in a subdomain $D'$ of $D$ subordinate to the point $(x_0,t_0)$.
\end{thm}
Some remarks are now in order:
\begin{rem}
The strong maximum and minimum principles involve only subdomains subordinate to the maximum/minimum point $(x_0,t_0)$ and fail for larger subdomains. In this direction, we refer to \cite{Landis} for some counterexamples in the linear case.
\end{rem}
\begin{rem}\label{ell}
Theorem \ref{mainSMP} implies the validity of strong maximum and minimum principles for the elliptic problem
\[
F(x,u,Du,D^2u)=0\text{ in }\Omega,
\]
where $\Omega$ is an open connected set in $\R^n$, namely any viscosity subsolution (resp. supersolution) attaining a nonnegative maximum (nonpositive minimum) in $\Omega$ is constant. It is sufficient to notice that any solution $u(x)$ of the elliptic equation solves the parabolic problem
\[
F(x,u,Du,D^2u)-\partial_t u=0\text{ in }\Omega\times(0,1).
\]
\end{rem}
\begin{rem}
Theorem \ref{simplemain} implies a strong comparison principle for the uniformly parabolic equation $F(D^2u)-\partial_t u=0$. This property means that if $u-v\leq 0$ in $\Omega_T$ and there exists a point $z\in \Omega_T$ such that $u-v=0$ in $\Omega_T$, then $u-v=0$ in $\Omega_T$ by the strong maximum principle. An example for strictly parabolic equations can be found in Theorem 4.1 of \cite{CLN}. This follows from the fact that $u-v$ solves a parabolic Pucci's extremal equation by means of Theorem 5.3 of \cite{CC}. 
\end{rem}
\begin{rem}
The strong maximum principle also implies a Tychonoff uniqueness principle for exponentially growing solutions to the uniformly parabolic equation
\[
F(x,t,D^2u)-\partial_t u=0
\]
by means of the three-curve property in \cite{Kovats}. Indeed, it is sufficient to note that if $u,v$ are two continuous solutions of the above equation in the strip $\R^n\times(0,T)$ with $u(x,0)=v(x,0)=g(x)$ and satisfying the growth condition
\[
|u|,|v|\leq c_1e^{c_2|x|^2},\ c_1,c_2>0
\]
uniformly in $t\in(0,T)$, then $w=u-v$ solves
\[
\mathcal{M}^+_{\lambda,\Lambda}(D^2w)-\partial_t w\geq0,
\]
by the results in \cite{CC}. Therefore, we can apply Theorem 1.1 in \cite{Kovats} and the argument of Theorem 1.3 therein along with the strong maximum principle to conclude $u\equiv v$ in $\R^n\times(0,T)$.
\end{rem}
\begin{rem}
If the strong minimum principle holds for \eqref{maineq} in $\overline{D}=\overline{\Omega}_T$, then we have the following conservation of positivity result: if $u\in C(\overline{\Omega}_T)$ with  $u(x,0)=u_0(x)\geq0$ and $u_0\neq0$, then it is strictly positive in the whole cylinder.
\end{rem}
\section{Examples}
Examples of operators to which the strong maximum and/or minimum principles proved in the previous section apply are those of the form
\[
E(x,t,u,\partial_t u,Du,D^2u)=a(x,t)G(Du,D^2u)+b(x,t)\cdot Du+c(x,t)u-\partial_t u
\]
when $b,c$ are bounded, $c\leq0$, $a>0$ continuous and bounded. $G$ can be of the following form:
\begin{itemize}
\item A Pucci's extremal operator or, more generally, a Bellman or an Isaacs operator defined respectively as 
\[
G(x,t,D^2u)=\sup_{\alpha}\mathrm{Tr}(A_\alpha(x,t) D^2u),\ G(D^2u)=\inf_{\alpha}\mathrm{Tr}(A_\alpha(x,t) D^2u);
\]
\[
G(x,t,D^2u)=\sup_{\beta}\inf_{\alpha}\mathrm{Tr}(A_{\alpha,\beta}(x,t) D^2u),
\]
where the linear operators are uniformly parabolic. These evolution operators satisfy both the strong maximum and minimum principles, as already discussed in \cite{DaLiopar} using a different proof.
\item The time-dependent normalized (1-homogeneous) $p$-Laplacian, $p>1$, so that $E$ corresponds to the evolution
\[
|Du|^{2-p}\mathrm{div}(|Du|^{p-2}Du)-\partial_t u=0.
\]
It is known that such operator belongs to the classes $\mathcal{S}$ defined above for $\lambda=\min\{1,p-1\},\Lambda=\max\{1,p-1\}$, $b=c=0$. We recall that the strong maximum and minimum principle fails for the standard parabolic $p$-Laplacian, cf. Example 2.6 in \cite{DaLiopar}.

\item The truncated (degenerate) Pucci's operator defined for $k<n$ by
\[
\mathcal{M}_{\lambda,\Lambda;k}^-(D^2u)=\lambda \sum_{\substack{i=0 \\ e_i>0}}^ke_i+\Lambda \sum_{\substack{i=0 \\ e_i<0}}^ke_i
\]
satisfies the strong maximum principle by comparison with the Pucci's extremal operator $\mathcal{M}^-_{\lambda,\Lambda}$. Indeed, a subsolution to $\mathcal{M}_{\lambda,\Lambda,k}^-(D^2u)-\partial_t u=0$ is also a subsolution to $\mathcal{M}^-_{\lambda',\Lambda'}(D^2u)-\partial_t u=0$, where $\lambda',\Lambda'$ are possibly different ellipticity constants, via the inequality
\[
\mathcal{M}_{\frac{n}{k}\lambda,\frac{n}{k}\Lambda;k}^-(X)\leq \mathcal{M}_{\lambda,\Lambda;n}^-(X)=\mathcal{M}_{\lambda,\Lambda}^-(X),\ X\in\mathcal{S}_n.
\]
It is known that the strong minimum principle for this operator fails even for the simpler case $\lambda=\Lambda=1$, where $\mathcal{M}_{\lambda,\Lambda;k}^-(D^2u)=\sum_{i=1}^ke_i(D^2u)$. Symmetrically, by duality one observes that $\mathcal{M}_{\lambda,\Lambda,k}^+(D^2u)-\partial_t u$ satisfies the strong minimum principle, but not the strong maximum principle (a counterexample is Example 4 of \cite{AB} for the case $\lambda=\Lambda=1$, $k<n$). 
\item Consider now the operator $G(D^2u)=\sum_{i=1}^n\arctan e_i(D^2u)$ and the equation
\[
\sum_{i=1}^n\arctan e_i(D^2u)-\partial_t u=0.
\]
Such an equation is called potential equation for the Lagrangian Mean Curvature Flow. It is known that when
\[
\Theta> (n-2)\frac{\pi}{2}\text{ and }\Theta=\sum_{i=1}^n\arctan e_i(D^2u)
\]
the operator is quasiconcave and uniformly parabolic, and hence any subsolution (resp. supersolution) belongs to the class $\underline{S}$ ($\overline{S}$). Therefore, the operator satisfies both the strong maximum and minimum principles. The same continues to hold for its elliptic counterpart via Remark \ref{ell}. 
\end{itemize}


%
\end{document}